\documentclass{article}
 \usepackage[margin=1in]{geometry} 
 \usepackage{enumerate}
\usepackage{amsmath,amsthm,amssymb,amsfonts, fancyhdr, comment, parskip, tikz, caption, subcaption, graphicx, bbold}
\usepackage[colorlinks]{hyperref}
\usetikzlibrary{graphs}
\usetikzlibrary{shadings}
\usetikzlibrary{shapes.geometric}
\usetikzlibrary{arrows}
\usetikzlibrary{calc, positioning, fit, shapes.misc}

\definecolor{purple}{RGB}{148,0,211}
\definecolor{green}{RGB}{50,205,50}

\newtheorem{theorem}{Theorem}[section]

\newtheorem{lemma}[theorem]{Lemma}

\newtheorem{claim}[theorem]{Claim}
\newtheorem{conjecture}[theorem]{Conjecture}

\newcommand{\prob}{\mathbb{P}}
\newcommand{\eps}{{\varepsilon}}

\newcommand{\E}{\mathbb{E}}

\newcommand{\mc}[1]{\mathcal{#1}}

\newcommand{\cH}{\mathcal{H}}

\newcommand{\rbrac}[1]{\left(#1\right)} 
\newcommand{\sbrac}[1]{\left[ #1\right]} 
\newcommand{\cbrac}[1]{\left\{ #1\right\}} 
\tikzstyle{arch} = [out=30, in=150]
\newcommand{\walk}{\text{walk}}
\newcommand{\walks}{\text{walks}}

\title{Edge-coloring $K_{n, n}$ with no 2-colored $C_{2k}$}
\author{Deepak Bal\footnote{Department of Mathematics, Montclair State University, Montclair, NJ. \texttt{deepak.bal@montclair.edu}.}
\and Patrick Bennett\footnote{Department of Mathematics, Western Michigan University, Kalamazoo, MI. \texttt{patrick.bennett@wmich.edu}. The research of P. Bennett was partially supported by Simons Foundation Grant \#848648.}  
}
\date{}

\begin{document}

\maketitle
\begin{abstract}
The generalized Ramsey number $r(G, H, q)$ is the minimum number of colors needed to color the edges of $G$ such that every isomorphic copy of $H$ has at least $q$ colors. In this note, we improve the upper and lower bounds on $r(K_{n, n}, C_{2k}, 3)$. Our upper bound answers a question of Lane and Morrison. For $k=3$ we obtain the asymptotically sharp estimate $r(K_{n, n}, C_6, 3) = \frac{7}{20} n + o(n)$. 
\end{abstract}
\section{Introduction}
For graphs $G, H$ we define an {\em $(H, q)$-coloring of $G$} to be an edge-coloring of $G$ where every isomorphic copy of $H$ has at least $q$ distinct colors. We define the {\em generalized Ramsey number} $r(G, H, q)$ to be the minimum number of colors needed in an $(H, q)$-coloring of $G$. In 2000 Axenovich, F\"uredi and Mubayi \cite{AFM} first defined $r(G, H, q)$ in full generality. They also gave some upper and lower bounds, mostly focusing on the case where $G=K_{n, n}$. This problem is a generalization of the one first proposed in 1975 by Erd\H{o}s and Shelah~\cite{erdos1975}, who were interested in the case where $G, H$ are complete $k$-uniform hypergraphs. In 1997 Erd\H{o}s and Gy\'arf\'as~\cite{EG} initiated the systematic study of the case where $G, H$ are complete graphs. Since then, this problem has generated significant interest, see e.g.~\cite{axenovich2000, BBHZ, BEHK,BCD, BCDP,CH1,CH2,CFLS,JM, LM, Mubayi1,mubayi2004}.

Axenovich, F\"uredi and Mubayi \cite{AFM} paid particular attention to $r(K_{n, n}, C_4, 3)$ and proved that
\[
\frac{2}{3}n \le r(K_{n, n}, C_4, 3) \le n+1,
\]
noting that ``improving this upper bound seems to be very hard.'' In a recent breakthrough, Joos and Mubayi \cite{JM} improved the upper bound to asymptotically match the lower bound. Their proof is a clever application of the conflict-free hypergraph matching method of Glock, Joos, Kim, K\"uhn, and Lichev~\cite{GJKKL}. This method gives sufficient conditions for a hypergraph to have an almost-perfect matching which avoids containing any member of a family of submatchings called ``conflicts'' (see Delcourt and Postle \cite{DP} for similar results). The main observation of Joos and Mubayi \cite{JM} is that one can obtain colorings for upper bounds on certain generalized Ramsey problems by encoding the coloring as a matching in a suitably defined hypergraph avoiding an appropriate family of conflicts. Joos, Mubayi and Smith \cite{JMS} further streamlined this method of coloring via conflict-free hypergraph matchings (giving the version we will use in this paper). This general method has recently been used to prove several more upper bounds on generalized Ramsey problems \cite{BBHZ, BCD, LM}. Lane and Morrison \cite{LM} and independently the present authors together with Heath and Zerbib \cite{BBHZ} showed that 
\[
r(K_n, C_k, 3) = \frac{1}{k-2} n + o(n)
\] for $k\ge 5$ (the case $k=3$ is easy and the case $k=4$ was done previously by Joos and Mubayi \cite{JM}).

The papers \cite{BBHZ, LM} also investigated the generalized Ramsey number $r(K_{n, n}, C_{2k}, 3)$ for $k \ge 3$ and obtained the bounds
\begin{equation}\label{eqn:previousbounds}
 \frac{1}{2(k-1)} n + o(n) \le r(K_{n, n}, C_{2k}, 3) \le \frac{2}{\left \lfloor \frac{1+\sqrt{8k-7}}{2} \right \rfloor ^2 -1 } n + o(n).  
\end{equation}
 Lane and Morrison \cite{LM} asked whether one can improve the upper bound by removing the integer floor in the denominator. In this note we improve both the upper and lower bounds in \eqref{eqn:previousbounds}, answering Lane and Morrison's question in the affirmative. In fact, our improved upper bound is better than \eqref{eqn:previousbounds} by a constant factor for large $k$. Our main theorem is as follows.

\begin{theorem}\label{thm:main}
For all $k \ge 4$ we have
\begin{equation}\label{eqn:main1}
    \frac{4 \sqrt{k^4 -6k^3+12k^2-9k+2} - 4k^2 + 12k -5}{2(k-2)}n \le r(K_{n,n}, C_{2k}, 3) \le \frac{3k-2}{2(k-1)(2k-1)}n +o(n),
\end{equation} 
and for $k=3$ we have
\begin{equation}\label{eqn:main2}
r(K_{n,n}, C_{6}, 3) = \frac{7}{20}n +o(n) . 
\end{equation}
\end{theorem}

For large $k$, the lower bound behaves like $(\frac12 + O(1/k) )\frac nk$ and the upper bound behaves like $(\frac34 + O(1/k) )\frac nk$. So the ratio of the upper bound to the lower bound is approximately $3/2$.  We briefly compare our new bounds \eqref{eqn:main1} to the previous bounds \eqref{eqn:previousbounds}. For $k=4$, \eqref{eqn:previousbounds} gives $r(K_{n, n}, C_{8}, 3) \in [0.166n, 0.25n]$. The new bounds \eqref{eqn:main1} give $r(K_{n, n}, C_{8}, 3) \in [0.227n, 0.239n]$. For large $k$, the upper bound in \eqref{eqn:main1} is approximately $3/4$ of the upper bound in \eqref{eqn:previousbounds}. The lower bound in \eqref{eqn:main1} is an improvement over the lower bound in \eqref{eqn:previousbounds} for all $k$, but unfortunately for large $k$ they are approximately the same. We conjecture that the upper bound in \eqref{eqn:main1} is the correct value (see the concluding remarks). 

We now describe the main idea which allows us to achieve this significant improvement in the upper bound. To prove an upper bound, we need to provide a coloring of $K_{n,n}$ without $2$-colored copies of $C_{2k}$. One might expect that each color class (essentially) looks like vertex disjoint copies of $K_{k-1, b}$ for some $b\ge k$. This would avoid all monochromatic $C_{2k}$. One can use older results (e.g., \cite{PipSpe}) to decompose almost all the edges of $K_{n,n}$ into such color classes. However, two of our colored bicliques may intersect in a way which produces a 2-colored $C_{2k}$. It it also possible that some set of more than two colored bicliques could interact in such a way that a 2-colored $C_{2k}$ is formed. One can use the method of conflict free matchings to avoid such interactions. We will see in Section \ref{sec:UB} when we define our hypergraph and conflict system that any matching in the hypergraph avoids 2-colored $C_{2k}$s arising from exactly two colored bicliques, and the conflict system consists of sets of more than two colored bicliques. This is the approach taken in \cite{LM} and \cite{BBHZ}. Our improvement comes from the observation that two colored bicliques of the form $K_{k-1, b}$ can actually intersect in a way which does not produce a $C_{2k}$. Indeed, they completely share their sides of size $k-1$. In our coloring, we initially set aside a small (subquadratic) collection of $(k-1)$-sets on each side of the partition of $K_{n,n}$. Each copy of $K_{k-1, b}$ that we pack must have its $k-1$ side in this collection, and  its $b$ side can only intersect an element of the collection in at most one vertex. By forcing the copies of $K_{k-1, b}$ to overlap in this way, we obtain a more efficient packing and hence use fewer colors. A similar idea (packing colored subgraphs that are forced to intersect in certain ways) was used by the second author, Cushman and Dudek \cite{BCD}.

The organization for the rest of the paper is as follows. We prove the lower bounds in Theorem \ref{thm:main} in Section \ref{sec:LB}, and the upper bounds in Section \ref{sec:UB}. Concluding remarks are in Section \ref{sec:conclusion}.


Throughout the paper we let $A \cup B$ be the bipartition of $K_{n, n}$. We use standard asymptotic notation, and for the rest of the paper all asymptotics are as $n \rightarrow \infty$.

\section{Lower bound for $k\ge 3$} \label{sec:LB}
Suppose we have a $(C_{2k}, 3)$-coloring of $K_{n, n}$ using $g$ colors.
\begin{claim}\label{clm:mindeg}
    For each connected monochromatic subgraph $C\subseteq K_{n, n}$, we have $\delta(C) \le k-1$. Furthermore if $\delta(C)=k-1$ then in the bipartition of $C$, one part has exactly $k-1$ vertices. 
\end{claim}
 \begin{proof}
     First suppose for contradiction that $\delta(C) \ge k$. Then $P_{2k} \subseteq C$ (see, for example \cite{Diestel}) which contradicts that we have a $(C_{2k}, 3)$-coloring. Thus we know $\delta(C) \le k-1$. Suppose now that $\delta(C) = k-1$ and, for contradiction, suppose that both parts of the bipartition of $C$ have at least $k$ vertices. Now we know $P_{2k-2} \subseteq C$ (similarly to the previous case). If the endpoints of this $P_{2k-2}$ are not adjacent, then each of them has a neighbor outside of the $P_{2k-2}$. Thus we can extend to a $P_{2k}$ and we have a contradiction. Otherwise we get a $C_{2k-2}$, and since $C$ is connected on at least $2k$ vertices we know there is some vertex $v$ not on the $C_{2k-2}$ such that $v$ has a neighbor in the $C_{2k-2}$, meaning that we have a $P_{2k-1} \subseteq C$. Denote this copy of $P_{2k-1}$ by $P=(v_1, \ldots v_{2k-1})$. Without loss of generality, say that $v_1, v_3, \ldots v_{2k-1} \in A$ and the other $k-1$ vertices of $P$ are in $B$. If $P$ can be extended to a copy of $P_{2k}$ then we are done, so assume to the contrary that $N(v_1)=N(v_{2k-1})=\{v_2, v_4, \ldots v_{2k-2}\}$. Since $P$ only has $k-1$ vertices in $B$, $C$ has an edge with one endpoint $v_{2i-1} \in \{v_3,v_5 \ldots v_{2k-3}\}$ and the other endpoint $w \in B \setminus V(P)$. But now we have the path $(w, v_{2i-1}, v_{2i-2}, \ldots , v_1, v_{2i}, v_{2i+1}, \ldots, v_{2k-1})$ which is a $P_{2k}$ so we are done. 
 \end{proof}

 For each monochromatic component $C$, we define a sequence $a_0(C), \ldots, a_{k-1}(C)$ in terms of the following stripping process on $C$. If $C$ does not contain a subgraph of minimum degree $k-1$, then we iteratively remove the lowest degree vertex until none are left.  Otherwise (i.e. when there is some subgraph of $C$ with minimum degree $k-1$), we remove vertices of smallest degree until we arrive at some $C' \subseteq C$ of minimum degree $k-1$. If $C'$ was disconnected then each component of $C'$ would have a $P_{2k-2}$ and so $C$ would have easily had a $P_{2k}$, which is a contradiction. Therefore $C'$ is connected. Now by Claim \ref{clm:mindeg}, $C'$ has exactly $k-1$ vertices in one of the parts of the bipartition, say $B$. At that point we continue removing vertices from $V(C') \cap A$ until none are left, and then finally we remove all the (now isolated) vertices in $V(C') \cap B$. Finally, we let $a_j(C)$ be the number of vertices that had degree $j$ at the time they were removed.   By counting vertices and edges, we have 
 \begin{equation*}
     \sum_{j=0}^{k-1} a_j(C) = |V(C)|, \qquad \sum_{j=0}^{k-1} j a_j(C) = |E(C)|.
 \end{equation*}
Therefore, letting  $\mc{F}$ is the set of all monochromatic components (in all colors), we have
\begin{equation}\label{eqn:1}
 \sum_{C \in \mc{F}}  \sum_{j=0}^{k-1} j a_j(C) = n^2
\end{equation}
and
 \begin{equation}\label{eqn:2}
    \sum_{C \in \mc{F}} \sum_{j=0}^{k-1} a_j(C) = 2ng.
 \end{equation}
 Let $\mc{F'}\subseteq \mc{F}$ be the set of components $C$ with $a_{k-1}(C) \ge k$. Note that if $C\in \mc{F}'$, then $a_0(C) \ge k-1$  by definition of the stripping process. 
Now we take $(k-2)$ times line \eqref{eqn:2} and subtract line \eqref{eqn:1} to obtain
\begin{align}
    2(k-2)ng - n^2 &= \sum_{C \in \mc{F}}  \Big( (k-2)a_0(C) + (k-3)a_1(C) + \ldots + 0a_{k-2}(C) - a_{k-1}(C) \Big)\nonumber\\
    &\ge \sum_{C\in\mc{F'}} \Big( (k-2)(k-1) - a_{k-1}(C) \Big).\label{eqn:3}
\end{align}
To see the inequality, note that for all $C\in\mc{F}\setminus \mc{F}'$, $a_{k-1}(C)=0$ or $a_{k-1}(C)=k-1$ and in the latter case, we also have $a_0(C)\ge k-1$.
Now letting $\ell:=(k-1)(k-2)$ and $A:=\sum_{C\in\mc{F}'}a_{k-1}(C)$, we have that
\begin{equation}\label{eqn:gLBwithF}
    2(k-2)ng \ge n^2 + \ell|\mc{F}'| - A.
\end{equation}

Now suppose we have two components $C, C' \in \mc{F}'$. $C$ contains a copy $K$ of $K_{k-1, a_{k-1}(C)}$. Likewise $C'$ contains a copy $K'$ of $K_{k-1, a_{k-1}(C')}$. Suppose further that $K$ and $K'$ share two vertices $u, v$ which are both in the larger parts of the bipartitions of $K$ and $K'$. More precisely, we are supposing that $u, v$ are contained in the part of $K$ containing $a_{k-1}(C)$ vertices, and in the part of the bipartition of $K'$ containing $a_{k-1}(C')$ vertices. Then $C \cup C'$ contains a $C_{2k}$ with only two colors, which is a contradiction. Crucially, this is true even if $C, C'$ share more vertices. This implies that

\begin{equation}\label{eqn:4}
   2 \binom n 2 \ge \sum_{C \in \mc{F}'} \binom{a_{k-1}(C)}{2} \ge |\mc{F}'| \binom{\frac{A}{|\mc{F}'|}}{2}, 
\end{equation}
where the second inequality follows by convexity. Rearranging, we get that
$2n^2|\mc{F}'|\ge A(A-|\mc{F}'|)$ and so
\begin{equation}
    |\mc{F}'|\ge \frac{A^2}{2n^2 + A}.
\end{equation}
Plugging this bound into \eqref{eqn:gLBwithF}, we get
\begin{equation}
    2(k-2)ng \ge n^2 + \ell\frac{A^2}{2n^2 + A} - A.
\end{equation}
The function $h(a) = 1+\ell\frac{a^2}{2+a} - a$ (for $a\ge 0$) has a minimum value of $3-4\ell + 4\sqrt{\ell(\ell-1)}$ at $a=2\rbrac{\sqrt{\frac{\ell}{\ell-1}} -1}$. Thus we have that 
\[2(k-2)ng \ge n^2h(A/n^2)\ge n^2\cdot(3-4\ell + 4\sqrt{\ell(\ell-1)})\] and so
\[g\ge \frac{3-4\ell + 4\sqrt{\ell(\ell-1)}}{2(k-2)}n\]
which is equivalent to the lower bound displayed in \eqref{eqn:main1}.

\subsection{Improved lower bound for $k=3$}
When $k=3$, the result proved above shows that $f(K_{n,n}, C_6, 3)\ge (2\sqrt{2} - \frac{5}{2})n\approx 0.3284n$.
We now show that  $f(K_{n,n}, C_6, 3) \ge \frac{7}{20}n = 0.35n$. This is the bound claimed in Theorem \ref{thm:main} which asymptotically matches the bound from the upper bound construction in Section \ref{sec:UB}. 

Note that when $k=3$, we have that $\ell=2$. 
The function $h(a) = 1+ \frac{2a^2}{2+a} -a$ is decreasing for $0\le a < \sqrt{8}-2$ and this interval includes $a=1/2$. Equation \eqref{eqn:1} implies that $A\le \frac{n^2}{2}$ and so we have $2ng\ge n^2h(A/n^2) \ge n^2h(1/2) = 7n^2/10$. So we conclude that $g \ge 7n/20$ as desired. This completes the proof of all lower bounds in Theorem \ref{thm:main}.

Note that this technique for improvement doesn't 
work when $k\ge 4$ since the corresponding inequality implied by \eqref{eqn:1} is that $A\le \frac{n^2}{k-1}$. The value $a = 1/(k-1)$ occurs to the right of the minimum of the function $h(a)$ and so this restriction on the domain is not useful.

\section{Upper bound for $k\ge 3$}\label{sec:UB}
 \subsection{Main tool: theorem of Joos, Mubayi and Smith \cite{JMS}}
    \label{section_matching_setup}
We now describe the framework for conflict-free hypergraph matching that we will use. Several proofs for generalized Ramsey numbers which made use of conflict-free matchings involved a ``two-stage'' coloring process. The second stage was required to deal with the fact that the results from \cite{GJKKL} only produce an almost-perfect matching, so they only color almost all the edges. 
The general ``black box'' theorem which we use was proved by Joos, Mubayi and Smith \cite{JMS} to incorporate both stages into one setup.   The description of the setup provided here is taken verbatim from \cite{JMS}. 

    Given two hypergraphs $ \mathcal{H} $ and $ \mathcal{C} $, say that $ \mathcal{C} $ is a \textit{conflict hypergraph for $ \mathcal{H} $} if $ V(\mathcal{C}) = E(\mathcal{H}) $. We call the edges of $ \mathcal{C} $ \textit{conflicts}.
    In this section, suppose that we are given the following setup:

    \begin{itemize}
        \item integers $ \ell \ge 2, d > 0 $ and real $ \varepsilon > 0 $;
        \item disjoint sets $ P, Q, R $ with $ d^{\varepsilon} \le |P| \le |P \cup Q| \le \exp(d^{\varepsilon^3}) $;
        \item hypergraph $ \mathcal{H}_1 $ whose edges consist of $ p \ge 1 $ vertices from $ P $ and $ q \ge 0 $ vertices from $ Q $;
        \item hypergraph $ \mathcal{H}_2 $ whose edges consist of a single vertex from $ P $ and $ r \ge 1 $ vertices from~$ R $;
        \item conflict hypergraph $ \mathcal{C} $ for $ \mathcal{H}_1 $;
        \item conflict hypergraph $ \mathcal{D} $ for $ \mathcal{H} := \mathcal{H}_1 \cup \mathcal{H}_2 $.
    \end{itemize}

    We assume that $ \mathcal{H} $ satisfies suitable degree conditions, and further that both $ \mathcal{C} $ and $ \mathcal{D} $ satisfy suitable boundedness conditions, all of which are specified in \ref{sec:Hcond}--\ref{sec:Dcond} in terms of $ d $ and $ \varepsilon $.
 
	\begin{theorem}[\cite{JMS}]
		\label{thm:blackbox}
		For $ p + q = k \ge 2 $, there exists $ \varepsilon_0 > 0 $ such that for all $ \varepsilon \in (0, \varepsilon_0) $, there exists $ d_0 $ such that given the above setup (including all conditions in \ref{sec:Hcond}--\ref{sec:Dcond}), the following holds for all $d \ge d_0$:  there exists a  $ P $-perfect matching $ \mathcal{M} \subseteq \mathcal{H} $ which contains none of the conflicts from $ \mathcal{C} \cup \mathcal{D} $. Furthermore, at most $ d^{-\varepsilon^4} |P| $ vertices of $ P $ belong to an edge in $ \mathcal{H}_2 \cap \mathcal{M} $.
	\end{theorem}

	\subsubsection{Notation}
	
	Write $ [i, n] = \{ i, \ldots, n \} $, so $ [n] = [1, n] $.
	Unless otherwise stated we identify hypergraphs with their edge sets, writing $ e \in \mathcal{G} $ to mean $ e \in E(\mathcal{G}) $.
    Given a set of vertices $ U \subseteq V(\mathcal{G}) $ in a hypergraph $ \mathcal{G} $, write $ d_{\mathcal{G}}(U) $ for the degree of $ U $ in $ \mathcal{G} $, that is the number of edges of $ \mathcal{G} $ containing $ U $; in the cases $ U = \{ u \} $ and $ U = \{ u, v \} $, where $ U $ consists of one or two vertices, we just write $ d_{\mathcal{G}}(u) $ and $ d_{\mathcal{G}}(u, v) $ respectively.
	We omit the subscript if $ \mathcal{G} $ is obvious from context.
	Write $ \Delta_j(\mathcal{G}) $ for the maximum degree $ d_{\mathcal{G}}(U) $ among sets $ U \subseteq V(\mathcal{G}) $ of $ j $ vertices.
 
    Given a subset of the vertices $ V \subseteq V(\mathcal{G}) $, write $\Delta_V(\mathcal{G})$ for the maximum degree $ d_{\mathcal{G}}(u) $ of any single vertex $ u \in V $, and similarly $ \delta_V(\mathcal{G}) $ for the minimum degree; assume $ V = V(\mathcal{G}) $ if not specified.
	Given $ j \in \mathbb{N} $, write $ \mathcal{G}^{(j)} := \{ E \in \mathcal{G} : |E| = j \} $ for the subhypergraph of $ \mathcal{G} $ containing only those edges of size $ j $.

	
	\subsubsection{Degree conditions on $ \mathcal{H} $}\label{sec:Hcond}
	
	We require that the hypergraph $ \mathcal{H} = \cH_1 \cup \cH_2 $ with $\cH_1,\cH_2\neq \emptyset$ satisfies the following conditions.

	\begin{enumerate}[(H1)]
		\item $ (1 - d^{-\varepsilon})d \le \delta_P(\mathcal{H}_1) \le \Delta(\mathcal{H}_1) \le d $; \label{cond_h_h1degree}
		\item $ \Delta_2(\mathcal{H}_1) \le d^{1 - \varepsilon} $; \label{cond_h_h1codegree}
		\item $\Delta_{R}(\mathcal{H}_2) \le d^{\varepsilon^4} \delta_{P}(\mathcal{H}_2) $; \label{cond_h_h2degree} 
		\item $ d(x, v) \le d^{-\varepsilon} \delta_{P}(\mathcal{H}_2) $ for each $ x \in P $ and $ v \in R $. \label{cond_h_h2codegree}
	\end{enumerate}
	
	This means that $ \mathcal{H}_1 $ is essentially regular for vertices in $ P $ and has small codegrees, although vertices in $ Q $ are allowed to have much lower (but not higher) degrees. Meanwhile in $ \mathcal{H}_2 $, every vertex in $ P $ must have degree at least a $ d^{-\varepsilon^4} $ proportion of the maximum degree in $ R $, and few edges in common with any particular vertex in $ R $. 
	
	\subsubsection{Boundedness conditions on $ \mathcal{C} $}\label{sec:Ccond}

	The conflicts of $ \mathcal{C} $ consist only of edges from $ \mathcal{H}_1$. We require that $\mc{C}$ satisfies the following conditions.
	\begin{enumerate}[(C1)]
		\item $ 3 \le |C| \le \ell $ for all $ C \in \mathcal{C} $; \label{cond_c1}
		\item $ \Delta(\mathcal{C}^{(j)}) \le \ell d^{j - 1} $ for all $ j \in [3, \ell] $; \label{cond_c2}
		\item $ \Delta_{j'}(\mathcal{C}^{(j)}) \le d^{j - j' - \varepsilon} $ for all $ j \in [3, \ell] $ and $ j' \in [2, j - 1] $. \label{cond_c3}
	\end{enumerate}
	
	\subsubsection{Boundedness conditions on $ \mathcal{D} $}\label{sec:Dcond}
	
	The conflicts of $ \mathcal{D} $ may consist only of edges from $ \mathcal{H}_2 $, or of two parts from $ \mathcal{H}_1 $ and $ \mathcal{H}_2 $, and so must satisfy a new set of conditions to be avoided.
    Write $ \mathcal{D}^{(j_1, j_2)} $ for the set of conflicts in $ \mathcal{D} $ consisting of $ j_1 $ edges from $ \mathcal{H}_1 $ and $ j_2 $ edges from $ \mathcal{H}_2 $.
	Similarly, 
    write $ \Delta_{j'_1, j'_2}(\mathcal{D}) $ for the maximum degree among sets $ F = F_1 \cup F_2 \subseteq \cH$ consisting of $ |F_1| = j'_1 $ edges from $ \mathcal{H}_1 $ and $ |F_2| = j'_2 $ edges from $ \mathcal{H}_2 $.
	Given a vertex $ x \in P $, write also $ \mathcal{D}_x $ for the set of conflicts in $ \mathcal{D} $ containing $ x $ in their $ \mathcal{H}_2 $-part, and likewise $ \mathcal{D}_{x, y} $ for those containing both $ x $ and $ y $.
	We require that $\mc{D}$ satisfies the following conditions for all $ x, y \in P $, and $ j_1 \in [0, \ell], j_2 \in [2, \ell] $.
	\begin{enumerate}[(D1)]
		\item $ 2 \le |D \cap \mathcal{H}_2| \le |D| \le \ell $ for each conflict $ D \in \mathcal{D} $; \label{cond_mixed_conflictsize}
		\item $ |\mathcal{D}^{(j_1, j_2)}_x| \le d^{j_1 + \varepsilon^4} \delta_{P}(\mathcal{H}_2)^{j_2} $; \label{cond_mixed_degree}
		\item $ \Delta_{j', 0}(\mathcal{D}^{(j_1, j_2)}_x) \le d^{j_1 - j' - \varepsilon} \delta_{P}(\mathcal{H}_2)^{j_2} $ for each $ j' \in [j_1] $; \label{cond_mixed_codegree1}
		\item $ |\mathcal{D}^{(j_1, j_2)}_{x, y}| \le d^{j_1 - \varepsilon} \delta_{P}(\mathcal{H}_2)^{j_2} $. \label{cond_mixed_codegree2}
	\end{enumerate}

\subsection{Definitions for our $\mc{H}_1, \mc{C}$, $\mc{H}_2$ and $\mc{D}$}
Our goal in this section is to define appropriate  $\mc{H}_1, \mc{C}$, $\mc{H}_2$ and $\mc{D}$ (as described above) such that a $P$-perfect matching in $\mc{H} = \mc{H}_1\cup\mc{H}_2$ which contains none of the conflicts from $\mc{C}\cup\mc{D}$   will correspond to a $(C_{2k},3)$-coloring of $K_{n,n}$. Recall again that throughout, we assume that $A\cup B$ is the bipartition of $K_{n,n}$

We start by obtaining a linear $(k-1)$-uniform hypergraph $\mc{S}_A$ on vertex set $A$ with some desirable properties. 
We say that a sequence $(S_1, S_2, \ldots, S_\ell)$ of (not necessarily distinct) edges  in $\mc{S}_A$ is a {\em $uv$-edge {\walk} of length $\ell$} if $u \in S_1, v \in S_\ell$, and $S_{i-1} \cap S_i \neq \emptyset$ for $i=2, \ldots, \ell$. 

Our proof will use some constant $\delta > 0$, and the conditions of Theorem \ref{thm:blackbox} will be satisfied for some $\eps >0$. These constants can be chosen so that 
\begin{equation}\label{eqn:constants}
    0<\delta \ll \eps \ll \frac 1k,
\end{equation}
where by $a \ll b$ we mean that there exists an increasing function $f$ such that our argument is valid if we choose $a < f(b)$.
\begin{lemma}\label{lem:S-exists}
     There exists a linear $(k-1)$-uniform hypergraph $\mc{S}_A$ on vertex set $A$ with the following properties:
    \begin{enumerate}[(S1)]
\item $n^{2-\delta} - n^{2-2\delta}\le |\mc{S}_A|\le n^{2-\delta} + n^{2-2\delta}$,
\item $(k-1)n^{1-\delta} - n^{1-2\delta}\le d_{\mc{S}_A}(v) \le (k-1)n^{1-\delta} + n^{1-2\delta}$ for all $v\in A$,

   \item  for any $a \in A$ and $\ell \in [1, 2k]$ there are $O(n^{\ell(1-\delta)})$ edge {\walks} of length $\ell$ having $a$ as an endpoint, and
    \item for any  $a, a' \in A$, and $ \ell \in [2, 2k]$, there are $O(n^{\ell(1-\delta)-1 })$ $aa'$-edge {\walks} of length $\ell$ that do not use any edge containing both $a, a'$.
\end{enumerate}
\end{lemma}

We will use the following concentration inequality due to McDiarmid~\cite{mcdiarmid} in our proof. 
\begin{theorem}[McDiarmid's inequality]\label{thm:McDiarmid}
Suppose $X_1,\ldots,X_m$ are independent random variables. Suppose $X$ is a real-valued random variable determined by $X_1,\ldots, X_m$ such that changing the outcome of $X_i$ changes $X$ by at most $b_i$ for all $i\in [m]$. Then, for all $t\geq 0$, we have \[\prob[|X-\E[X]|\geq t]\leq 2\exp\left(-\frac{2t^2}{\sum_{i\in [m]}b_i^2}\right).\]
\end{theorem}

\begin{proof}[Proof of Lemma \ref{lem:S-exists}]
   We start with a ``complete'' $(k-1)$-uniform linear hypergraph $K$ with $\frac{\binom n2}{\binom{k-1}{2}}$ edges on vertex set $A$ (these are known as $(n,k-1,2)$-Steiner systems and  are known to exist from the work of Wilson in the 70's \cite{Wil1, Wil2, Wil3}). Of course, there are some divisibility conditions on $n$ in order for such a hypergraph $K$ to exist, but we can easily take care of this by rounding $n$ up to satisfy the divisibility conditions. We then let $\mc{S}_A$ be a random subgraph of $K$  where each edge is retained with probability $(k-1)(k-2)/n^{\delta}$. Note that the expected degree of a vertex $v$ in $\mc{S}_A$ is $(k-1)n^{1-\delta}$. 

We will explicitly justify (S4) using McDiarmid's inequality. Properties (S1) and (S2) are similar but easier. Property (S3) follows directly from (S2). 

To justify (S4), we first fix $a, a' \in A$ and let $\ell=2$. For our application of McDiarmid's inequality, our independent variables $X_i$ will be indicators, one for the presence of each edge in $K$. $X$ will be the number of $aa'$-edge {\walks} of length 2 in $\mc{S}_A$.  Note that in $K$ there are precisely $n-2$ $aa'$-edge {\walks} of length 2 that do not use the edge containing both $a, a'$. Thus $\E[X]=\Theta(n^{1-2\delta})$. For our values $b_i$, the only nonzero values correspond to edges containing $a$ or $a'$, there are $O(n)$ such edges, and the corresponding values $b_i$ are $O(1)$. Thus we have
\[\prob[X \ge 2\E[X]] \le \prob[|X-\E[X]|\geq \E[X]]\leq 2\exp\left(-\frac{\Theta(n^{2-4\delta})}{O(n)}\right)= \exp\left( -\Omega(n^{1-4\delta})\right).\]
We can choose $\delta< 1/4$, and so by the union bound over $O(n^2)$ choices for $a, a'$ we have that (S4) holds for $\ell=2$ with probability tending to 1 as $n \rightarrow \infty$. Now for $\ell \ge 3$ we can use (S3) with $\ell-2$ in place of $\ell$ and then we use (S4) in the $\ell=2$ case. Thus we obtain our hypergraph $\mc{S}_A$ with properties (S1)--(S4). 
\end{proof}

Let $E_A:= \binom A2 \setminus \bigcup_{
S \in \mc{S}_A} \binom S2 = \{aa': a, a' \in A \mbox{ and } \{a, a'\} \not\subseteq S \; \forall \; S \in \mc{S}_A\}$ be the set of pairs in $A$ which are not contained in any hyperedge of $\mc{S}_A$. Analogously we obtain a hypergraph $\mc{S}_B$ on vertex set $B$, and the set of pairs $E_B$ not contained in any element of $\mc{S}_B$. 

\subsubsection{Definition of $\mc{H}_1$}
For our application of Theorem \ref{thm:blackbox}, we will let $\mc{H}_1$ be the following hypergraph. The vertex set will be $V(\mc{H}_1)= P \cup Q$ where
\begin{equation*}
    P:=E(K_{n, n}), \qquad Q:= \left\{ v_c: v \in V(K_{n, n}), c \in W \right\} \cup E_A \cup E_B.
\end{equation*}
We now describe the edges of $\mc{H}_1$. Each such edge will correspond to colored copy of $K = K_{k-1, 2k-1}$ in $K_{n, n}$. Letting $Y$ be the part of the bipartition of $K$ with $k-1$ vertices and $Z$ be the other part, we require that $Y \in \mc{S}_A \cup \mc{S}_B$ and $\binom Z2 \subseteq E_A \cup E_B$. For the rest of this section, $K$ will always represent a copy of $K_{k-1, 2k-1}$, and when we say that $K$ has bipartition $Y \cup Z$ we assume $|Y|=k-1, |Z|=2k-1$. For each such $K$ and color $c\in C$, $\mc{H}_1$ will have an edge 
\begin{equation*}
    e_{K, c} = E(K) \cup \{v_c: v \in V(K)\} \cup \binom Z2.
\end{equation*}

\subsubsection{Definition of $\mc{C}$}
Next we describe our conflict system $\mc{C}$. First we define a conflict system $\mc{C}'$. Consider a set of distinct bicliques $K_1, \ldots, K_j$ isomorphic to $K_{k-1, 2k-1}$, where $4 \le j\le 2k$ is even. Assume that each of these bicliques has its partite set of size $k-1$ in $\mc{S}_A \cup \mc{S}_B$ and that $V(K_i) \cap V(K_{i+1}) \neq \emptyset$ for $i=1\ldots j$ (subscripts taken modulo $j$). Now for any two distinct colors $c, c'$ the conflict system $\mc{C}'$  will have a conflict $C=\{e_{K_1, c}, e_{K_2, c'}, \ldots, e_{K_j, c'}\}$ (where the colors $c, c'$ alternate). If $C$ is not a matching in $\mc{H}$, of course we do not include it in $\mc{C}'$. We now define $\mc{C}$ to be the set of minimal conflicts in $\mc{C}'$. Note that any $\mc{C}$-free matching is also $\mc{C}'$-free. 

\begin{claim}\label{clm:legit}
    A $\mc{C}$-free matching in $\mc{H}_1$ corresponds to a partial coloring of $E(K_{n, n})$ with no 2-colored $C_{2k}$. 
\end{claim}

\begin{proof}
    
Suppose there is some 2-colored $C_{2k}$. This cycle can be broken into monochromatic paths, each of which must come from some $e_{K, c}$. In particular we have some sequence $\sigma=(e_{K_1, c}, e_{K_2, c'}, \ldots, e_{K_L, c'})$ alternating in the colors $c, c'$ such that any two consecutive $K_{i}, K_{i+1}$ share a vertex (in the {\em cyclic ordering}, i.e.~ where we consider the first and last to be consecutive, or equivalently let $K_{L+1}:=K_1$). To show that we have a conflict, we will ``clean'' our sequence $\sigma$ so that the bicliques are all distinct, and this cleaning will maintain the property that consecutive bicliques intersect. We define an {\em annoyance} to be a maximal subsequence of consecutive elements from $\sigma$ which just alternates between two bicliques and starts and end with the same bicliques, i.e., is of odd length. To {\em contract} an annoyance is an operation on $\sigma$ where we will skip from the beginning to the end of the annoyance to produce a shorter sequence $\sigma'$. For example, if  $K_3 = K_5 = K_7 $ and $K_2 \neq K_4 = K_6 \neq K_8$ then we would have an annoyance $(e_{K_3, c}, e_{K_4, c'}, \ldots,  e_{K_7, c})$. If say $L=8$ and we contract this annoyance, we arrive at the new sequence $\sigma' = (e_{K_1, c}, e_{K_2, c'}, e_{K_3, c}, e_{K_8, c'})$. Note that in $\sigma'$, consecutive bicliques still intersect (in our example, $K_8$ intersects $K_7 = K_3$). We observe the following. In the original sequence $\sigma=(e_{K_1, c}, \ldots, e_{K_L, c'})$, there must be at least two distinct bicliques of color $c$ (and two of color $c'$). Indeed, if say $K_1$ was the only biclique of color $c$ here, then every $c'$-colored biclique $K_2, \ldots, K_L$ must share at least two vertices with $K_1$, and so the small sides $Y_1 = Y_2 =\ldots = Y_L$ all coincide. The union of these bicliques cannot contain a $C_{2k}$.
Therefore $\sigma$ has at least four bicliques, and therefore there are two bicliques (say pick any two of the same color) with disjoint $Y$-sets. Also, we observe that the contraction operation does not affect the set of distinct $Y$-sets used by bicliques in the sequence. Indeed, the two bicliques in an annoyance share a $Y$-set and one of them remains after contraction.
Now contract all of the annoyances to get say $\sigma''=(e_{K_1', c}, e_{K_2', c'}, \ldots, e_{K_{L'}, c'})$, where we have renamed some of the bicliques $K_1, \ldots K_L$ and possibly reindexed so that the color alternation starts with $c$. It is still possible that we have some repetition of bicliques in $\sigma''$ but we cannot have $K_i' = K_{i+2}'$ (since $\sigma''$ does not have any annoyance). Take a maximal sequence of the form $(K_{1}', \ldots, K_{i}')$ of distinct bicliques (where maximality means that either $i=L'$ since all the bicliques in $\sigma''$ are distinct,  or else our sequence cannot be extended because $K_{i+1}' = K_{i'}'$ for some $i' \le i$. In the first case we are done, since the set of colored bicliques in $\sigma''$ is a conflict. Note that the length of $\sigma''$ is at least $4$ since there are two disjoint $Y$-sets in the bicliques from $\sigma''$. In the second case we have $K_{i+1}' = K_{i'}'$ and we have the conflict using the bicliques $K_{i'}', K_{i'+1}', \ldots K_i'$ (with appropriate alternating colors $c, c'$). In this last case note that our conflict does have at least four colored bicliques since we have contracted all annoyances.
\end{proof}

\subsubsection{Definition of $\mc{H}_2$}
Now we describe $\mc{H}_2$.  $R$ is a set of pairs $(e, c)$. We let $R$ be a set of 
$n^2\cdot n^{1-\delta}$ pairs $(e,c)$ where $e\in E(K_{n,n})$ and each $c$ represents a color from a new set, $W'$ of $n^{1-\delta}$ colors (disjoint from $W$).
The vertex $e\in P$ is adjacent to all pairs $(e,c)$ in $R$. So $\mc{H}_2$, the bipartite graph on $P\cup R$, is a disjoint union of stars.

\subsubsection{Definition of $\mc{D}$}

Finally, we describe the conflict system $\mc{D}$. First we define a conflict system $\mc{D}'$. Consider a set $D$ consisting of some edges from $\mc{H}_2$ and possibly also some edges from $\mc{H}_1$. As above, each edge from $\mc{H}_1$ corresponds to a colored copy of $K_{k-1, 2k-1}$. Each edge of $\mc{H}_2$ corresponds to a colored edge of $K_{n, n}$ in the obvious way. If $D$ is a matching in $\mc{H}_1\cup\mc{H}_2$ and the colored edges in $K_{n, n}$ corresponding to $D$ would contain a 2-colored cycle of length at most $2k$, then $D \in \mc{D}'$. Now we let $\mc{D}$ be the set of minimal conflicts in $\mc{D}'$. $\mc{D}$ includes some conflicts that are entirely in $\mc{H}_2$, which correspond to cycles of length at most $2k$ colored with at most 2 colors.  If a conflict in $\mc{D}$ does have any edges in $\mc{H}_1$, then this conflict corresponds to some set of bicliques $K_1, \ldots, K_{j_1}$ of the same color say $c\in W$, together with $j_1$ vertex disjoint paths, all monochromatic in some other color $c'\in W'$, with one such path connecting each pair of consecutive bicliques (in the cyclic ordering). Given Claim \ref{clm:legit}, it is easy to see that any $\mc{C} \cup \mc{D}$-free $P$-perfect matching in $\mc{H}$ corresponds to a $(C_{2k}, 3)$-coloring of $K_{n, n}$.

We now embark on checking the conditions for Theorem \ref{thm:blackbox}.

\subsection{$\mc{H}$ conditions}
Here we check the conditions from Section \ref{sec:Hcond}. We will let 
\begin{equation}\label{eqn:ddef}
    d = \frac{3k-2}{(2k-1)!} n^{2k-\delta} + \kappa n^{2k-2\delta}
\end{equation}
for some large constant $\kappa$.

\begin{enumerate}[(H1)]
		\item Let $e \in E(K_{n, n})$. Then
       \[
       d_{\mc{H}_1}(e) = 2 \rbrac{(k-1)n^{1-\delta} + O\rbrac{n^{1-2\delta}}} \sbrac{\binom{n}{2k-2} + O\rbrac{n^{2k-2-\delta}}} |W| = \frac{3k-2}{(2k-1)!} n^{2k-\delta} + O\rbrac{n^{2k-2\delta}}
       \]
    {\bf Explanation:} Given $e=ab \in K_{n, n}$ where $a \in A,  b \in B$, we need to choose an appropriate colored copy $K$ of $K_{k-1, 2k-1}$. We first choose one of two possible orientations (i.e. whether $K$ has $k-1$ vertices in $A$ or in $B$). Let us assume that $K$ has $k-1$ vertices in $A$ (the other case being similar). Then we choose a hyperedge of $\mc{S}_A$ containing $a$. Then we choose a set $Z'$ of $2k-2$ vertices from $B$ such that $\binom{Z' \cup \{b\}}{2} \subseteq E_B$. \\
    Now consider a vertex $v_c$ where $v \in V(K_{n, n}), c \in W$. Then
    \begin{align*}
           d_{\mc{H}_1}(v_c) &=\rbrac{(k-1)n^{1-\delta} + O\rbrac{n^{1-2\delta}}}\sbrac{\binom{n}{2k-1} + O\rbrac{n^{2k-1-\delta}}} + n^{2-\delta} \sbrac{\binom{n}{2k-2} + O\rbrac{n^{2k-2-\delta}}} \\
           &= \frac{3k-2}{(2k-1)!} n^{2k-\delta} + O\rbrac{n^{2k-2\delta}}
    \end{align*}
    {\bf Explanation:} Given $v$ and $c$ we have to determine an appropriate $K$. Assume without loss of generality that $v \in A$. If $v\in Y$ then we have to choose a hyperedge $Y \in \mc{S}_A$ containing $v$, and then choose a set $Z$ such that $\binom{Z}2 \subseteq E_B$. This explains the first term. If $v\in Z$ then we have to choose a hyperedge $Y \in \mc{S}_B$, and then choose a set $Z'$ such that $\binom{Z' \cup \{v\}}2 \subseteq E_A$. This explains the second term. \\
    Finally we consider a vertex $aa' \in E_A$ (the case of a pair in $E_B$ is similar).  Then similarly to the above calculations we have
    \begin{align*}
           d_{\mc{H}_1}(aa') &= n^{2-\delta} \sbrac{\binom{n}{2k-3} + O\rbrac{n^{2k-3-\delta}}}|W|= \frac{3k-2}{(2k-1)!} n^{2k-\delta} + O\rbrac{n^{2k-2\delta}}.
    \end{align*}
Thus (H1) holds for any $\eps < \delta$ and $\kappa$ sufficiently large. 

\item We fix two vertices of $\mc{H}_1$. There are several cases. We will describe in detail one case. Consider the case where we fix $ab \in E(K_{n, n})$ and some $v_c$ for some $v \in V(K_{n, n}), c \in W$. Then we need to determine a $K$ containing $a, b, v$ and having color $c$. It is possible that $v =a$ or $v=b$, but even in that case we are choosing some edge from $S_A$ or $S_B$ containing one of our fixed vertices and then $2k-2$ vertices (and we already know the color $c$) so we have $O(n^{2k-1-\delta})$ choices. For the other cases (i.e. fixing a pair of other types vertices from $\mc{H}_1$), we always find that we always fix at least three vertices from $K_{n, n}$ or else two vertices from 
$K_{n, n}$ plus a color, and there are always $O(n^{2k-1-\delta})$ choices to complete $K$. Thus (H2) is satisfied for some small $\eps, \delta$ (see \eqref{eqn:constants}).

\item In $\mc{H}_2$, the degree of every vertex in $R$ is 1, and the degree of every vertex in $P$ is $|W'|=n^{1-\delta}$. Thus (H3) easily holds.

\item In $\mc{H}_2$, $d(x,v)\le 1$ for each $x\in P$ and $v\in R$. On the other hand, $d^{-\eps}\delta_P(\mc{H}_2) \ge \Omega(n^{1-\eps}/n^{\eps(2k-\delta)})$ which tends to infinity when $\eps$ is sufficiently small. Thus (H4) also holds.
	\end{enumerate}

\subsection{$\mc{C}$ conditions}
Here we check the conditions from Section \ref{sec:Ccond}
\begin{enumerate}[(C1)]
\item Each $C\in\mc{C}$ contains at least 4 and at most $2k$ hyperedges from $\mc{H}_1$.
\newcommand{\Cj}{\mc{C}^{(j)}}
\item Note that $\mc{C} = \bigcup_j\Cj$ where we recall that $\mc{C}^{(j)}$ is the set of all conflicts of size $j$. Let $j\in \cbrac{4,6,\ldots, 2k}$. Fix an edge $e_{K_1,c}$ from $\mc{H}_1$. To bound $\Delta(\Cj)$, we need to bound the number of ways in which $e_{K_1,c}$ can be completed to a conflict of size $j$. Our conflict is of the form $\{e_{K_1, c}, e_{K_2, c'}, \ldots , e_{K_j, c'}\}$ where any two consecutive $K_i, K_{i+1}$ (subscript taken modulo $j$) share a vertex. Since the conflict must be a matching in $\mc{H}_1$ by definition, the only way that two consecutive $K_i, K_{i+1}$ could share more than one vertex is if their small parts completely overlap.

Letting $K_i$ have bipartition $Y_i \cup Z_i$, we will first choose the $Y_i$, recalling that each $Y_i$ is an edge of $\mc{S}_A$ or $\mc{S}_B$. Suppose $x \in V(K_1) \cap V(K_2)$ and $y \in V(K_1) \cap V(K_j)$. Note that $x\neq y$ since $K_2,K_j$ are distinct bicliques of the same color (so they are disjoint). By (S4), there are $O(n^{(j-1)(1-\delta)-1})$ choices for the $(Y_2, \ldots, Y_j)$ that would be $xy$-edge {\walks}. Now we count possibilities that are not $xy$-edge {\walks}. We can always partition $(Y_2, \ldots, Y_j)$ into maximal subsequences $(Y_i, Y_{i+1}, \ldots Y_{i'})$ which are edge {\walks}. Suppose this partition has $p\ge 2$ edge {\walks} whose lengths are $j_1, \ldots, j_p$, where $j_1 + \ldots + j_p = j-1$. First we will choose $p -\mathbb{1}_{x \in Y_2} - \mathbb{1}_{y\in Y_j}$ vertices, intending that each will be an endpoint of one of our edge {\walks}, where $x$ (resp. $y$) will be the endpoint of the first (resp. last) edge {\walk} if possible. Now to choose the edges in the $i$th edge {\walk} we have $O(n^{j_i(1-\delta)})$ choices by property (S3). Thus the number of ways to choose the $\{Y_2, \ldots, Y_j\}$ is 
\begin{equation}\label{eqn:Ychoice}
   O\rbrac{n^{p -\mathbb{1}_{x \in Y_2} - \mathbb{1}_{y\in Y_j} +  (j-1)(1-\delta)}}. 
\end{equation} We point out here that in the case discussed above where $Y_2,\ldots, Y_j$ formed an $xy$-edge {\walk}, we get the same count as in \eqref{eqn:Ychoice} by setting $p=1$ and both indicator functions to 1.\\
Now we choose the $Z_2, \ldots, Z_j$. We choose them in groups according to the maximal edge {\walks} we used to choose our $Y$ sets. For the first group $Z_2, \ldots Z_{j_1+1}$ we choose $j_1 (2k-1) -\mathbb{1}_{x \in Z_2}$ vertices. For $Z_{j_1+2}, \ldots Z_{j_1 + j_2+1}$ we choose $j_2(2k-1) - 1$ vertices (the ``$-1$'' is because $Z_{j_1 + 1}$ must share a vertex with $Y_{j_1} \cup Z_{j_1}$). Similarly, for $1<q<p$, assuming we have chosen the $Z$ sets for our first $q-1$ edge {\walks}, when we choose $Z$ sets for the $q$th edge {\walk} we have to choose $j_q(2k-1)-1$ vertices. Finally for the last edge {\walk} we choose $j_p (2k-1) -1-\mathbb{1}_{y \in Z_j}$ vertices. Thus the number of ways to choose our $Z$ sets is 
\begin{equation}\label{eqn:Zchoice}
   O\rbrac{n^{(j-1)(2k-1) -(p-1)-\mathbb{1}_{x \in Z_2} -\mathbb{1}_{y \in Z_j}}}.
\end{equation}
Multiplying \eqref{eqn:Ychoice} times \eqref{eqn:Zchoice} times $O(n)$ to pick a color $c'$, our bound on the number of conflicts is 
\[
O\rbrac{n^{(j-1)(2k-\delta)}} = O(d^{j-1}).
\]
Thus (C2) is satisfied for some $\ell = O(1)$.

\item Fix some set $J' \subseteq \mc{H}$ with $|J'|=j'$. We would like to bound the possible number of extensions $J \supseteq J'$ to conflicts $J \in \mc{C}$ with $|J|=j$. So we bound the number of ways to complete some $J$, say $J= \{e_{K_1, c}, \ldots, e_{K_j, c'}\}$, given our fixed $J' \subseteq J$. We assume that two bicliques in $J$ intersect if they are consecutive in the cyclic ordering indicated by the subscripts. Partition the elements of $J'$ into maximal subsequences of consecutive elements, and say there are $p$ parts in this partition. To determine the rest of $J$ we need to choose $p$ sequences of colored bicliques where consecutive bicliques intersect (and where the first and last bicliques intersect with the appropriate bicliques from $J'$).  To choose such a sequence of length $q$, we would sequentially choose $q-1$ bicliques each containing some already fixed vertex, and then the last biclique contains two fixed vertices. There are $O(n^{2k})$ ways to choose a biclique with one fixed vertex, and $O(n^{2k-1})$ ways with two fixed vertices, so there are $O(n^{2kq-1})$ ways to complete our sequence of length $q$. The sum of the length of all our sequences of bicliques to be chosen is $j-j'$. Putting it all together, there are $O\rbrac{n^{2k(j-j')-p}}$ ways to choose all the bicliques of $J$. If $J'$ has at least one colored biclique in each color $c, c'$ then we do not choose a color to form $J$ and since $p\ge 1$ the number of extensions $J$ is $ O\rbrac{n^{2k(j-j')-1}} \le d^{j-j' - \eps}$ since we can choose $\eps, \delta$ to be small (and they can depend on $k$). If $J'$ only has one color, then we must have $p \ge 2$ and we have $O(n)$ choices for the other color, so the number of extensions $J$ is $ O\rbrac{n^{2k(j-j')-2} \cdot n}\le d^{j-j' - \eps}$ since $\eps$ is small.
\end{enumerate}
\subsection{$\mc{D}$ conditions}
Here we check the conditions from Section \ref{sec:Dcond}

	\begin{enumerate}[(D1)]
		\item For all $D \in \mc{D}$, we have $2 \le |D\cap \mc{H}_2| \le |D|\le 2k$.
        \item First assume $j_1 \ge 1$.  To determine a conflict in $\mathcal{D}^{(j_1, j_2)}_x$, we first choose two colors in $O(n^2)$ ways, then we choose $j_1$ bicliques in $O(n^{j_1(2k+1)-p})$ ways, where $p\le 2$ is the number of vertices our edge $x \in E(K_{n, n})$ has in common with our bicliques. Then we choose $j_2-1$ additional edges, whose union (together with $x$) is a set of vertex disjoint paths connecting our bicliques. Note that these edges form $j_1$ paths. So we choose $j_2-j_1+p-2$ vertices to be the internal vertices of our paths (taking care not to include the $2-p$ that are endpoints of $x$). This gives
        \[
        |\mathcal{D}^{(j_1, j_2)}_x| \le O\rbrac{n^{2+j_1(2k+1)-p +j_2-j_1+p-2}} \le O\rbrac{n^{2kj_1 +j_2}} \le  d^{j_1 + \varepsilon^4} \delta_{P}(\mathcal{H}_2)^{j_2}
        \]
        where we have used that $\delta_P(\mc{H}_2) = n^{1-\delta}$, and where the last inequality is justified since the power of $n$ on the right hand side is $(2k-\delta)(j_1+\eps^4) + (1-\delta)j_2$ which we can make larger than $2kj_1+j_2$ by choosing $\delta$ sufficiently small depending on $k, \eps$. 

        In the case when $j_1=0$, we have $j_2 \le 2k$. In other words, this conflict consists of a cycle colored with $2$ colors from $W'$.  We first choose 2 colors in $|W'|^2 = n^{2-2\delta}$ ways. We then choose $j_2-2$ vertices other than those in $x$ to form a cycle.  
        This gives 
        \[
        |\mathcal{D}^{(j_1, j_2)}_x| \le O\rbrac{n^{j_2-2\delta}} \le d^{\varepsilon^4}\delta_P(\mc{H}_2)^{j_2}
        \]
        where the last inequality is justified since the power of $n$ on the right hand side is $(2k-\delta)\eps^4 + (1-\delta)j_2$ which we can make larger than $j_2-2\delta$ by choosing $\delta \ll \eps \ll 1/k$ as in \eqref{eqn:constants}. 

        \item We fix an edge $x$ and a set $J'$ of $j' \ge 1$ bicliques of the same color. We bound the number of conflicts $J_1 \cup J_2 \in \mc{D}$ with $x$ incident with $J_2$. We also require that $J_i \subseteq \mc{H}_i$, $|J_i|=j_i$, $J' \subseteq J_1$. There are $O(n^{(2k+1)(j_1-j')-p})$ ways to choose the rest of our bicliques
        if the bicliques in $J_1 \setminus J'$ are to use $p \le 2$ of the endpoints of our edge $x$. 
        Then we have $O(n)$ ways to choose another color and $O(n^{j_2 - j_1 })$ ways to connect our bicliques with at least $j_1-1$ paths (since $x$ could be the entirety of one of our paths) whose total length is $j_2-1$. Altogether the number of possible conflicts $J_1 \cup J_2$ is at most (under the asssumption $j'+p \ge 2$
        \[
        O\rbrac{n^{(2k+1)(j_1-j')-p+1 +j_2 - j_1  }} =  O\rbrac{n^{2k(j_1-j')-p+1-j' +j_2  }} \le d^{j_1 - j' - \eps}\delta_P(\mc{H}_2)^{j_2},
        \]
        where the final inequality follows since $\eps, \delta$ are small. Now we handle the remaining case, where $j'=1, p=0$. In this case our edges in $J_2$ (except $\{x\}$) comprise at least $j_1$ paths and so there are $O(n^{j_2 - 1 - j_1})$ ways to choose them. In this case we get the bound
        \[
        O\rbrac{n^{(2k+1)(j_1-1)+1 +j_2 -1- j_1  }} =  O\rbrac{n^{2k(j_1-1)-1 +j_2  }} \le d^{j_1 - 1 - \eps}\delta_P(\mc{H}_2)^{j_2},
        \]
        since $\eps, \delta$ are small.

        \item First assume $j_1 \ge 1$.  To determine a conflict in $\mathcal{D}^{(j_1, j_2)}_{x, y}$, we first choose two colors in $n^2$ ways, then we choose $j_1$ bicliques in $O(n^{j_1(2k+1)-p})$ ways, where $p\le 4$ is the number of vertices  $x \cup y $ has in common with our bicliques. Then we choose $j_2-2$ additional edges, whose union (together with $x, y$) is a set of vertex disjoint paths connecting our bicliques. Note that these edges form $j_1$ paths. So we choose at most $j_2-j_1+p-3$ vertices to be the internal vertices of our paths (taking care not to include the $|x \cup y|-p \ge 3-p$ that are in $x\cup y$). This gives
        \[
        |\mathcal{D}^{(j_1, j_2)}_{x,y}| \le O\rbrac{n^{2+j_1(2k+1)-p +j_2-j_1+p-3}} \le O\rbrac{n^{2kj_1 +j_2-1}} \le  d^{j_1 - \varepsilon} \delta_{P}(\mathcal{H}_2)^{j_2}
        \]
        since $\eps, \delta$ are small. 

        In the case when $j_1=0$, we have $j_2 \le  2k$. In other words, this conflict consists of cycle colored with $2$ colors from $W'$.  We first choose 2 colors in $|W'|^2 = n^{2-2\delta}$ ways. We then choose at least $j_2-3$ vertices other than those in $x \cup y$ to form a cycle.  
        This gives 
        \[
        |\mathcal{D}^{(j_1, j_2)}_{x, y}| \le O\rbrac{n^{2+j_2-3}} = O\rbrac{n^{j_2-1}} \le d^{ - \varepsilon} \delta_{P}(\mathcal{H}_2)^{j_2}
        \]
         since $\eps, \delta$ are small. 
	\end{enumerate}

\section{Concluding remarks}\label{sec:conclusion}

There are many cases of generalized Ramsey numbers that are still wide open. Most related to our current paper, it would be interesting to study $r(K_{n, n}, C_{2k}, q)$ for other values of $q >3$. 

We conjecture that our upper bound in Theorem \ref{thm:main} is tight in general. 
\begin{conjecture}
    For all $k \ge 2$ we have
    \[r(K_{n,n}, C_{2k}, 3) = \frac{3k-2}{2(k-1)(2k-1)}n +o(n).\]
\end{conjecture}
The conjecture is known to be true for $k=2$ by Joos and Mubayi \cite{JM} and now $k=3$ by Theorem \ref{thm:main}. We provide a little motivation by proving the conjecture under some (believable, but admittedly strong) additional assumptions. 

\begin{theorem}
    Suppose we have a $(C_{2k}, 3)$-coloring of $K_{n, n}$ such that every monochromatic component is of the form $K_{k-1, b}$ where $b \ge k$ (and different components can have different values of $b$). Then the coloring uses at least $\frac{3k-2}{2(k-1)(2k-1)}n$ colors. 
\end{theorem}

\begin{proof}
    Let $F$ be the set of monochromatic components, and for $j=1, \ldots, |F|$ let $b_j$ be the value of $b$ for the $j$th  component. Let $B:=\sum_{j} b_j$. Then we have
    \begin{equation}\label{eqn:conc1}
        n^2 = (k-1)B, \qquad 2ng \ge (k-1)|F|+B.
    \end{equation}
    Indeed, the first equation is just counting edges of $K_{n, n}$ and the second is counting the vertices in monochromatic components. Now since the $b_j \ge k$, we cannot have two monochromatic components $C_i, C_j$ that share two vertices on the sides of their bipartitions having $b_i, b_j$ vertices respectively. Therefore 
    \begin{equation}\label{eqn:conc2}
        n^2 \ge \sum_j \binom{b_j}{2} \ge \frac12 B \rbrac{\frac{B}{|F|} -1}
    \end{equation}
    where the second inequality follows by convexity. Now \eqref{eqn:conc1} gives
    \begin{equation}\label{eqn:conc3}
        g \ge \frac{(k-1)|F|+B}{2n} = \frac{(k-1)|F|+\frac{n^2}{k-1}}{2n}
    \end{equation}
and \eqref{eqn:conc2} gives
\begin{equation}\label{eqn:conc4}
        |F| \ge \frac{B^2}{B+2n^2} = \frac{n^2}{(k-1)(2k-1)}.
    \end{equation}
    Now \eqref{eqn:conc3} and \eqref{eqn:conc4} give
    \[
    g \ge \frac{3k-2}{2(k-1)(2k-1)}n.
    \]
\end{proof}
Note that the above can be relaxed a bit to say that ``almost'' all the edges of $K_{n, n}$ are colored using the components of the type we described (allowing for $o(n^2)$ edges to be colored some other way). Indeed, this would just mean that $B = \frac{n^2}{k-1}+o(n^2)$ and we would get the bound $ g \ge \frac{3k-2}{2(k-1)(2k-1)}n+o(n).$


\begin{thebibliography}{10}

\bibitem{axenovich2000}
M.~Axenovich.
\newblock A generalized {R}amsey problem.
\newblock {\em Discrete Mathematics}, 222(1-3):247--249, 2000.

\bibitem{AFM}
M.~Axenovich, Z.~F\"uredi, and D.~Mubayi.
\newblock On generalized {R}amsey theory: the bipartite case.
\newblock {\em J. Combin. Theory Ser. B}, 79(1):66--86, 2000.

\bibitem{BBHZ}
D.~Bal, P.~Bennett, E.~Heath, and S.~Zerbib.
\newblock Generalized {Ramsey} numbers of cycles, paths, and hypergraphs.
\newblock {\em arXiv preprint arXiv:2405.15904}, 2024.

\bibitem{BEHK}
J.~Balogh, S.~English, E.~Heath, and R.~A. Krueger.
\newblock Lower bounds on the {E}rd{\H{o}}s--{G}y{\'a}rf{\'a}s problem via color energy graphs.
\newblock {\em Journal of Graph Theory}, 103(2):378--409, 2023.

\bibitem{BCD}
P.~Bennett, R.~Cushman, and A.~Dudek.
\newblock The generalized {Ramsey} number \(f(n,5,8)=\tfrac{6}{7}n+o(n)\).
\newblock {\em arXiv preprint arXiv:2408.01535}, 2024.

\bibitem{BCDP}
P.~Bennett, R.~Cushman, A.~Dudek, and P.~Pra{\l}at.
\newblock {The {E}rd\H os-{G}y\'arf\'as} function {$f(n,4,5)=\frac{5}{6}n+o(n)$}---so {G}y\'arf\'as was right.
\newblock {\em J. Combin. Theory Ser. B}, 169:253--297, 2024.

\bibitem{CH1}
A.~Cameron and E.~Heath.
\newblock A $(5,5)$-colouring of ${K}_n$ with few colours.
\newblock {\em Combinatorics, Probability and Computing}, 27(6):892--912, 2018.

\bibitem{CH2}
A.~Cameron and E.~Heath.
\newblock New upper bounds for the {E}rd{\H{o}}s--{G}y\'arf\'as problem on generalized {R}amsey numbers.
\newblock {\em Combinatorics, Probability and Computing}, 32(2):349–362, 2023.

\bibitem{CFLS}
D.~Conlon, J.~Fox, C.~Lee, and B.~Sudakov.
\newblock The {E}rd{\H{o}}s--{G}y{\'a}rf{\'a}s problem on generalized {R}amsey numbers.
\newblock {\em Proceedings of the London Mathematical Society}, 110(1):1--18, 2015.

\bibitem{DP}
M.~Delcourt and L.~Postle.
\newblock Finding an almost perfect matching in a hypergraph avoiding forbidden submatchings.
\newblock {\em arXiv preprint arXiv:2204.08981}, 2022.

\bibitem{Diestel}
R.~Diestel.
\newblock {\em Graph theory}, volume 173 of {\em Graduate Texts in Mathematics}.
\newblock Springer, Berlin, sixth edition, [2025] \copyright 2025.

\bibitem{erdos1975}
P.~Erd\H{o}s.
\newblock Problems and results on finite and infinite graphs.
\newblock {\em Recent Advances in Graph Theory (Proc. Second Czechoslovak Sympos., Prague, 1974)}, pages 183--192, 1975.

\bibitem{EG}
P.~Erd{\H{o}}s and A.~Gy{\'a}rf{\'a}s.
\newblock A variant of the classical {R}amsey problem.
\newblock {\em Combinatorica}, 17(4):459--467, 1997.

\bibitem{GJKKL}
S.~Glock, F.~Joos, J.~Kim, M.~K{\"u}hn, and L.~Lichev.
\newblock Conflict-free hypergraph matchings.
\newblock {\em Journal of the London Mathematical Society}, 109(5):e12899, 2024.

\bibitem{JM}
F.~Joos and D.~Mubayi.
\newblock {Ramsey} theory constructions from hypergraph matchings.
\newblock {\em Proceedings of the American Mathematical Society}, 152(11):4537--4550, 2024.

\bibitem{JMS}
F.~Joos, D.~Mubayi, and Z.~Smith.
\newblock Conflict‑free hypergraph matchings and coverings.
\newblock {\em arXiv preprint arXiv:2407.18144}, 2024.

\bibitem{LM}
A.~Lane and N.~Morrison.
\newblock {Generalized {Ramsey} Numbers via Conflict-Free Hypergraph Matchings}.
\newblock {\em arXiv preprint arXiv:2405.16653}, 2024.

\bibitem{mcdiarmid}
C.~McDiarmid et~al.
\newblock On the method of bounded differences.
\newblock {\em Surveys in combinatorics}, 141(1):148--188, 1989.

\bibitem{Mubayi1}
D.~Mubayi.
\newblock Edge-coloring cliques with three colors on all 4-cliques.
\newblock {\em Combinatorica}, 18(2):293--296, 1998.

\bibitem{mubayi2004}
D.~Mubayi.
\newblock An explicit construction for a {R}amsey problem.
\newblock {\em Combinatorica}, 24(2):313--324, 2004.

\bibitem{PipSpe}
N.~Pippenger and J.~Spencer.
\newblock Asymptotic behavior of the chromatic index for hypergraphs.
\newblock {\em Journal of Combinatorial Theory, Series A}, 51(1):24--42, May 1989.

\bibitem{Wil1}
R.~M. Wilson.
\newblock An existence theory for pairwise balanced designs, {I}: Composition theorems and morphisms.
\newblock {\em Journal of Combinatorial Theory, Series A}, 13(2):220--245, 1972.

\bibitem{Wil2}
R.~M. Wilson.
\newblock An existence theory for pairwise balanced designs, {II}: The structure of {PBD}‑closed sets and the existence conjectures.
\newblock {\em Journal of Combinatorial Theory, Series A}, 13(2):246--273, 1972.

\bibitem{Wil3}
R.~M. Wilson.
\newblock An existence theory for pairwise balanced designs, {III}: Proof of the existence conjectures.
\newblock {\em Journal of Combinatorial Theory, Series A}, 18:71--79, 1975.

\end{thebibliography}
\bibliographystyle{abbrv}

\end{document}